\documentclass[11pt]{amsart}

\usepackage{amsmath}
\usepackage{amstext}
\usepackage{amssymb}
\usepackage{amsthm}
\usepackage{enumerate}
\usepackage{bbm}
\usepackage{comment}
\usepackage[USenglish]{babel}
\usepackage[utf8]{inputenc}
\usepackage[T1]{fontenc}
\usepackage{dsfont}
\usepackage{graphicx}%,txfonts}

\DeclareSymbolFont{extraup}{U}{zavm}{m}{n}
\DeclareMathSymbol{\varheart}{\mathalpha}{extraup}{86}

\theoremstyle{plain}
\newtheorem{thm}{Theorem}
\theoremstyle{definition}
\newtheorem{defi}[thm]{Definition}
\newtheorem{cor}[thm]{Corollary}
\newtheorem*{cor*}{Corollary}
\newtheorem*{lem*}{Lemma}
\newtheorem*{thm*}{Theorem}
\newtheorem{lem}[thm]{Lemma}
\newtheorem{question}[thm]{Question}

\newtheorem{clm}[thm]{Claim}

\newtheorem{rem}[thm]{Remark}

\renewcommand{\phi}{\varphi}

\date{\today}
\begin{document}

\title[Parallel non-linear iterations]{Parallel non-linear iterations}

\author{\"Omer Faruk Ba\u{g}}
\address{Institute of Mathematics, Kurt G\"odel Research Center, University of Vienna, Augasse 2-6, UZA 1 - Building 2, 1090 Wien, Austria}
\email{oemer.bag@univie.ac.at}

\author{Vera Fischer}
\address{Institute of Mathematics, Kurt G\"odel Research Center, University of Vienna, Augasse 2-6, UZA 1 - Building 2, 1090 Wien, Austria}
\email{vera.fischer@univie.ac.at}

\thanks{\emph{Acknowledgments.}: The authors would like to thank the Austrian Science Fund (FWF) for the generous support through Grant Y1012-N35.}



\begin{abstract}
Developing a system of parallel non-linear iterations, we
establish the consistency of $\mathfrak{b}<\mathfrak{s}<\mathfrak{d}<\mathfrak{c}$ where $\mathfrak{b}, \mathfrak{d}, \mathfrak{c}$ are arbitrary subject to the known ZFC restrictions and $\mathfrak{s}$ is regular. By evaluating other invariants we achieve also the constellations $\mathfrak{b}<\mathfrak{r}<\mathfrak{d}<\mathfrak{c}$, $\mathfrak{b}<\mathfrak{e}<\mathfrak{d}<\mathfrak{c}$ and $\mathfrak{b}<\mathfrak{u}<\mathfrak{d}<\mathfrak{c}$.
\end{abstract}

\maketitle

\section{Introduction}

Cardinal characteristics of the continuum are well studied in many research and survey articles like \cite{blass} and \cite{vnDwn}. In this article we study the the bounding, dominating and the splitting numbers $\mathfrak{b}, \mathfrak{d}, \mathfrak{s}$. Our method allows us to evaluate the reaping and the evasion numbers $\mathfrak{r}$ and  $\mathfrak{e}$.

It is well-known that $\mathfrak{s}, \mathfrak{b} \leq \mathfrak{d}$ and that $\mathfrak{s}$ and $\mathfrak{b}$ are independent. The consistency of $\mathfrak{s} < \mathfrak{b}$ was shown by J. Baumgartner and P. Dordal in 1985 by preserving a strong splitting property. In 1984 S. Shelah showed the consistency of $\aleph_1 = \mathfrak{b} < \mathfrak{s} = \aleph_2$, later generalized by the second author and J. Steprans to an arbitrary regular uncountable $\kappa$. The more general result $\mathfrak{b} = \kappa < \mathfrak{s}= \lambda$ was shown in \cite{brndl/fschr} using the method of a matrix iteration, which was introduced by A. Blass and S. Shelah in 1989 to prove the relative consistency of $\mathfrak{u} < \mathfrak{d}$, where the ultrafilter number $\mathfrak{u}$ denotes the minimal size of a base for a non-principal ultrafilter on $\omega$. The authors of \cite{brndl/fschr} systematized and further developed this method and gave the consistency of $\mathfrak{b} = \mathfrak{a} = \kappa < \mathfrak{s} = \lambda$ (where $\mathfrak{a}$ denotes the minimal size of a maximal almost disjoint family of subsets of $\omega$) and  $\mu < \mathfrak{b} = \kappa < \mathfrak{a} = \mathfrak{s} = \lambda$ above a measurable cardinal $\mu$. The inequality $\mathfrak{s} < \mathfrak{d}$ holds already in Cohen's model. The relative consistency of $\aleph_1 = \mathfrak{b} < \mathfrak{e} = \kappa$ was shown in \cite{brndl/shlh}. The method of non-linear iteration was introduced in \cite{cmmngs} in order to control the general invariants $\mathfrak{b}(\kappa), \mathfrak{d}(\kappa)$ and $\mathfrak{c}(\kappa)$ simultaneously where $\kappa$ is a regular cardinal. Making a large cardinal assumption on $\kappa$, the authors showed in \cite{ofb/vf} the relative consistency of $\mathfrak{s}(\kappa) = \kappa^+ < \mathfrak{b}(\kappa)< \mathfrak{d}(\kappa) < \mathfrak{c}(\kappa)$. The method in the article \cite{ofb/vf} gives also an alternative proof of $\mathfrak{s} = \aleph_1 < \mathfrak{b}< \mathfrak{d} < \mathfrak{c}$ without random forcing.

Here we want to address the question whether $\mathfrak{b} < \mathfrak{s}< \mathfrak{d} < \mathfrak{c}$ is consistent. After revisiting the basic definitions in the second section, we point out in the third that we can use a strict partial order as an index set for a non-linear forcing iteration. Then we give the definition of a non-linear forcing iteration modified for our purposes in this article, define some relations between models of ZFC and prove their preservation. In the last section we achieve the following main result combining the method of matrix iterations, non-linear iterations and adding restricted Hechler reals as in \cite{brndl/fschr}.

\begin{thm*}
	If $\beta, \lambda, \delta, \mu$ are infinite cardinals with $\omega_1 \leq \beta = cf(\beta) \leq \lambda = cf(\lambda) \leq \delta \leq \mu$ and $cf(\mu) > \omega$, then there is a cardinal preserving generic extension where $\beta = \mathfrak{b} \land \lambda = \mathfrak{r} = \mathfrak{s} = \mathfrak{e} \land \delta = \mathfrak{d} \land \mu = \mathfrak{c}$ holds.
\end{thm*}

Using suitable forcing notions we evaluated also $\mathfrak{r}$ and $\mathfrak{e}$ in our model, establishing the consistency of $\mathfrak{b} < \mathfrak{e}< \mathfrak{d} < \mathfrak{c}$ and $\mathfrak{b} < \mathfrak{r}< \mathfrak{d} < \mathfrak{c}$. The constellation $\mathfrak{b} < \mathfrak{u} < \mathfrak{d} < \mathfrak{c}$ follows by a known result. Note that this refines also the above mentioned consistency of $\aleph_1 = \mathfrak{b} < \mathfrak{e}$.

\section{Preliminaries}

%$\varheartsuit$

%In this section we want to show the following consistency: If $\kappa_0$ is a measurable cardinal and $\beta, \lambda, \delta, \mu$ are cardinals above $\kappa_0$ such that $\beta = cf(\beta) \leq \lambda \leq cf(\delta) \leq \delta \leq \mu$ and $cf(\mu)> \omega$, then there is a model $\beta = \mathfrak{b} \land \lambda = \mathfrak{a} \land \delta = \mathfrak{d} \land \mu = \mathfrak{c}$.

\begin{defi}
	Let $f$ and $g$ be functions from $\omega$ to $\omega$, i.e. $f, g \in {}^{\omega}\omega$.
	
	\begin{enumerate}
		
		\item Then $g$ eventually dominates $f$, denoted by $f <^* g$, if $\exists n < \omega$ $\forall m > n$ $f(m) < g(m)$.
		
		\item A family $\mathcal{F} \subseteq {}^{\omega}\omega$, is dominating if $\forall g \in {}^{\omega}\omega $ $\exists f \in \mathcal{F}$ such that $g<^*f$.
		
		\item A family $\mathcal{F} \subseteq {}^{\omega}\omega$ is unbounded if $\forall g \in {}^{\omega}\omega$ $\exists f \in \mathcal{F}$ such that $f \not<^*g$.
		
		\item $\mathfrak{b}$ and $\mathfrak{d}$ denote the generalized bounding and dominating numbers respectively:
		
		\begin{center}
			$\mathfrak{b} = min \{ \vert \mathcal{F}\vert : \mathcal{F} \subseteq {}^{\omega}\omega, \mathcal{F}$ is unbounded$\}$,
			
			$\mathfrak{d} = min \{ \vert \mathcal{F}\vert : \mathcal{F} \subseteq {}^{\omega}\omega, \mathcal{F}$ is dominating$\}$.
		\end{center}
		
		\item Finally $\mathfrak{c}= 2^\omega$.
	\end{enumerate}
\end{defi}

%\begin{defi}
%Let $x, y  \in [\omega]^\omega$.
%\begin{enumerate}

%\item The sets $x$ and $y$ are almost disjoint if $\vert x \cap y\vert < \aleph_0$. 

%\item An infinite family $\mathcal{A} \subseteq [\omega]^\omega$ is almost disjoint if any two elements in $\mathcal{A}$ are almost disjoint. An almost disjoint family is maximal almost disjoint (mad) if it is maximal with respect to inclusion.

%\item The almost disjointness number $\mathfrak{a}$ is the minimal size of a mad family:

%\begin{center}
%$\mathfrak{a} = min\{|\mathcal{A}|: \mathcal{A} \subseteq [\omega]^\omega, \mathcal{A}$ is mad$\}.$
%\end{center}
%\end{enumerate}
%\end{defi}

\begin{defi}
	Let $x, y  \in [\omega]^\omega$.
	\begin{enumerate}
		
		\item The set $x$ splits $y$ if $\vert x \cap y\vert = |y \setminus x| = \aleph_0$. 
		
		\item A family $\mathcal{S} \subseteq [\omega]^\omega$ is splitting if $\forall a \in [\omega]^\omega \exists s \in \mathcal{S}: s$ splits $a$.
		
		\item A family $\mathcal{R} \subseteq [\omega]^\omega$ is reaping if it is not split by a single real.
		
		\item The splitting number $\mathfrak{s}$ is the minimal size of a splitting family, and the reaping number $\mathfrak{r}$ is the minimal size of a reaping family:
		
		\begin{center}
			$\mathfrak{s} = min\{|\mathcal{S}|: \mathcal{A} \subseteq [\omega]^\omega, \mathcal{A}$ is splitting$\}.$
			
			$\mathfrak{r} = min\{|\mathcal{R}|: \mathcal{A} \subseteq [\omega]^\omega, \mathcal{A}$ is reaping$\}.$
		\end{center}
	\end{enumerate}
\end{defi}

\begin{defi}
	Let $D \in [\omega]^\omega$ and $f \in \omega^\omega$:
	\begin{enumerate}
		
		\item A pair $\pi = (D, \langle \pi_n: n \in D \rangle)$ where each $\pi_n: {}^n \omega \to \omega$ is called a predictor. 
		
		\item A predictor $\pi = (D, \langle \pi_n: n \in D \rangle)$ predicts $f$ if $\{n \in D: \pi_n(f\upharpoonright n) = f(n)\}$ is cofinite. Otherwise we say that $f$ evades $\pi$.
		
		\item The evasion number $\mathfrak{e}$ is the minimal size of a family of reals, which are not predicted by a single predictor.
	\end{enumerate}
\end{defi}

\begin{defi}
Let $\mathcal{F}$ be a non-principial ultrafilter on $\omega$.

\begin{enumerate}
  \item A subset $\mathcal{G} \subseteq \mathcal{F}$ is a base for $\mathcal{F}$ if for every $F \in \mathcal{F}$ there is a $G \in \mathcal{G}$ such that $G \subseteq F$. In this case we say that $\mathcal{G}$ generates $\mathcal{F}$.

\item The ultrafilter number is defined as follows: $\mathfrak{u} = min\{\vert \mathcal{G}\vert: \mathcal{G}$ generates a non-principial ultrafilter$\}$.
\end{enumerate}
\end{defi}

Some of the relations between the above mentioned invariants are as follows: $\aleph_1 \leq \mathfrak{b} = cf(\mathfrak{b}) \leq cf(\mathfrak{d}) \leq \mathfrak{d} \leq \mathfrak{c}$, $\mathfrak{b} \leq \mathfrak{r} \leq \mathfrak{u}$, $\mathfrak{s} \leq \mathfrak{d}$, $cf(\mathfrak{c}) > \omega$, $\mathfrak{e} \leq \mathfrak{d}$ and $\mathfrak{s}$ and $\mathfrak{b}$ are independent (see e.g. \cite{blass} for more).
%$\mathfrak{b} \leq \mathfrak{a} \leq \mathfrak{c}$,
%and $\mathfrak{a}$ and $\mathfrak{d}$ are independent

\begin{defi}\cite{brndl/shlh}\label{PR}
	$\mathbb{P}$ consists of triples $(d, \pi, F)$ where $d \in {}^{<\omega}2$, $\pi = \langle \pi_n: n \in d^{-1}(\{1\}) \rangle$ and $\pi_n: {}^n\omega \to \omega$ and $F \in [\omega^\omega]^{<\aleph_0}$ with $\forall f, g \in F~ [f \not= g \to~ min\{n: f(n) \not= g(n)\}< |d|]$.
	
	$(d', \pi', F') \leq (d, \pi, F)$ iff $d' \supseteq d$, $\pi' \supseteq \pi$, $F' \supseteq F$ and  $\forall f \in F~ \forall n \in (d')^{-1}(\{1\})\setminus d^{-1}(\{1\})~[\pi'_n(f\upharpoonright n) = f(n)]$.
\end{defi}

$\mathbb{P}$ is $\sigma$-centered, hence has the c.c.c.. $\mathbb{P}$ adds a predictor, which predicts every ground model real.

\begin{defi}\label{MAT} Let $U$ be an ultrafilter on $\omega$. The Mathias forcing with respect to $U$, denoted $\mathbb{M}_U$, consists of pairs $(a, A)  \in [\omega]^\omega \times U$ such that $max(a) < min(A)$ with the following extension relation: $(a, A) \leq (b, B)$ iff $b \subseteq a$, $a \backslash b \subseteq B$ and $A \subseteq B$.
\end{defi}

The $\mathbb{M}_U$ has the c.c.c. and adds an unsplit real.

\begin{defi}\label{EV}
	$\mathbb{E}$ consists of pairs $(s, F)$ where $s \in {}^{<\omega}\omega$, $F = \{\pi_D: \pi_D = (D, \langle \pi^D_i: i \in D\rangle)$ is a predictor$\}$ and $|F| < \aleph_0$.
	
	$(t, G) \leq (s, F)$ iff $t \supseteq s$, $G \supseteq F$ and $\forall i \in |t| \backslash |s| ~ \forall \pi_D \in F~ [i \in D \to t(i) \not= \pi^D_i(t\upharpoonright i)]$.
\end{defi}
$\mathbb{E}$ is $\sigma$-centered, hence has the c.c.c.. $\mathbb{E}$ adds a function from $\omega$ to $\omega$, which is not predicted (evaded) by any ground model predictor.

\begin{defi}\label{def_H}
	The Hechler forcing notion is defined as the set $\mathbb{H} = \{(s,f): s \in \omega^{<\omega}, f \in {}^\omega \omega\}$. The extension relation is given by:
	
	\centerline{$(t,g) \leq_{\mathbb{H}} (s,f)$ iff $s \subseteq t \land \forall n \in \omega ~[g(n)\geq f(n)] \land \forall i \in dom(t)\setminus dom(s)~ [t(i) > f(i)]$.}
	
	If $A$ is a collection of reals then $\mathbb{H}(A) = \{(s,f): s \in \omega^{<\omega}, f \in A\}$, equipped with the same extension relation, is the restriction of the Hechler forcing to $A$.
\end{defi}

Clearly the later notion adds a real dominating only the elements in $A$.

\begin{defi}\label{compl.em}
	If $(Q, \leq_Q, 1_Q) $ and $(P, \leq_P, 1_P)$ are forcing posets, then $i : Q \to P$ is called a complete embedding if:
	
	\begin{enumerate}
		\item $i(1_Q) = 1_P$
		
		\item $\forall q, q' \in Q [q \leq_Q q' \to i(q) \leq_P i(q')]$
		
		\item $\forall q, q' \in Q [q \perp_Q q' \leftrightarrow i(q) \perp_P i(q')]$
		
		\item if $A \subseteq Q$ is a maximal antichain in $Q$, then $i(A)$ is a such in $P$.
	\end{enumerate}

	%If for an $i: Q \to P$ (1), (2), (3) holds and $i$ satisfies
	%\begin{enumerate}
	%\item $i(Q)$ is a dense subset of $P$\\
	%then $i$ is a dense embedding.
	%\end{enumerate}
	
\end{defi}

\begin{defi}
	Let $P, Q$ be forcing posets and $i: Q \to P$ satisfy properties (1), (2), (3) of Definition \ref{compl.em}. Then for an element $p \in P$, $p' \in Q$ is called a reduction of $p$ to $Q$ iff $\forall q \in Q [i(q) \perp p \to q \perp p']$.
\end{defi}
A proof of the next lemma can be found in \cite{kun}.
\begin{lem} \label{1.5}
	If $i: Q \to P$ satisfies (1),(2),(3) of Definition \ref{compl.em}, then $i$ is a complete embedding iff $\forall p \in P$ there is a reduction of $p$ to $Q$.
\end{lem}

\section{Generalized bounding and dominating for strict partial orders}

Let $P$ be a set. Let $P_\sqsubset =(P, \sqsubset)$ be a strict partial order (transitive, irreflexive) and let $P_ \sqsubseteq = (P, \sqsubseteq)$ be a partial order (transitive, reflexive, anti-symmetric). $\mathfrak{b}(P_ \sqsubseteq)$ and $\mathfrak{d}(P_ \sqsubseteq)$ are defined in \cite{cmmngs} for partial orders.

\begin{defi}
	1) We call $U \subseteq P$ unbounded if it is not dominated (bounded) by a single element in $P$, i.e. $\forall p \in P \exists q \in U q \not\sqsubset p$. Define $\mathfrak{b}(P_\sqsubset)$ to be the least size of an unbounded family of $P$.
	
	2) A subset $D$ of $P$ is dominating if every element in $P$ is dominated (bounded) by some element in $D$, i.e. $\forall p \in P \exists q \in D p \sqsubset q$. Define $\mathfrak{d}(P_\sqsubset)$ to be the least size of a dominating subset of $P$.
\end{defi}

\begin{rem} If $P_\sqsubset =(P, \sqsubset)$ and $P_ \sqsubseteq = (P, \sqsubseteq)$ are as above, then
	\begin{itemize}
		
		\item If $\forall p, q \in P~[p\sqsubset q \to p \sqsubseteq q]$, then $\mathfrak{d}(P_\sqsubset) \geq \mathfrak{d}(P_\sqsubseteq)$.
		
		\item If $\forall p \in P \exists q \in P ~ [p \sqsubset q]$ and $\forall p, q, r \in P~[p\sqsubseteq q \sqsubset r \to p \sqsubset r]$, then $\mathfrak{d}(P_\sqsubset) \leq \mathfrak{d}(P_\sqsubseteq)$.
		
		\item If $\forall p, q \in P~ [p \not\sqsubseteq q \to p \not\sqsubset q]$, then $\mathfrak{b}(P_\sqsubset) \leq \mathfrak{b}(P_\sqsubseteq)$.
		
		\item If $\forall p, q \in P \big[ [p \not\sqsubset q] \to \exists r \in P ~ [p \sqsubset r \not \sqsubseteq q] \big]$, then $\mathfrak{b}(P_\sqsubset) \geq \mathfrak{b}(P_\sqsubseteq)$.
		
	\end{itemize}
\end{rem}

\begin{proof}
	We prove the the first two items. For the first let $D$ be a witness for $\mathfrak{d}(P_\sqsubset)$ and let $p \in P$. Because $D$ is dominating, $\exists q \in D~ p \sqsubset q$. By assumption $p\sqsubset q \to p \sqsubseteq q$. So $D$ is a witness for $\mathfrak{d}(P_\sqsubseteq)$.
	
	For the second let $D$ be a witness for $\mathfrak{d}(P_\sqsubseteq)$ and let $p \in P$. Because $D$ is dominating, $\exists q \in D~ p \sqsubseteq q$. By assumption $\exists q' \in P~ q \sqsubset q'$. By assumption $p \sqsubseteq q \sqsubset q' \to p \sqsubset q'$. So $\{q': q \in D\}$ is a witness for $\mathfrak{d}(P_\sqsubset)$ of the same size.
\end{proof}

\begin{cor}
	Let $\beta, \delta $ be infinite cardinals such that $cf(\beta) = \beta \leq cf(\delta)$. Let $P = ([\delta]^{< \beta}, \subseteq)$ and $Q = ([\delta]^{< \beta}, \subset)$. Then $\beta = \mathfrak{b}(P) = \mathfrak{b}(Q)$ and $\delta = \mathfrak{d}(P) = \mathfrak{d}(Q)$. Further for any regular $\kappa$ we have $\mathfrak{b}(({}^\kappa \kappa, <^*)) = \mathfrak{b}(({}^\kappa \kappa, \leq^*))$ and $\mathfrak{d}(({}^\kappa \kappa, <^*)) = \mathfrak{d}(({}^\kappa \kappa, \leq^*))$
\end{cor}

Note that $Q$ is well-founded, so applying Theorem 1 and Theorem 2 in \cite{cmmngs} with the strict poset $Q$ we can obtain a model where $\mu = \mathfrak{c} \geq \delta = \mathfrak{d} \geq \beta = \mathfrak{b}$ where $\beta, \delta$ and $\mu$ are arbitrary uncountable cardinals with $cf(\mu) > \omega$, $cf(\beta) = \beta \leq cf(\delta)$, $\delta \leq \mu$. Further note that $\forall a \in Q\linebreak rk_Q(a) = type(a, \in)$.

\section{Preservation Theorems}
%Recall \cite{brndl/fschr} $\S 2$ (Adding a mad family).

%\begin{defi}\label{bs1}
%Let $\kappa$ be an ordinal, then the forcing poset $\mathbb{H}_\kappa$ consists of all finite functions $p: \kappa \times \omega \to 2$, with $dom(p) = F_p \times n_p$ where $F_p \in [\kappa]^{<\omega}, ~ n_p \in \omega$. The extension relation is given as $q \leq p $ iff $p \subseteq q$ and $ \forall i \in n_q \setminus n_p ~|q^{-1} \cap F_p \times \{i\}| \leq 1$.
%\end{defi}

%If $G$ is a $\mathbb{H}_\kappa$-generic, then we consider the family $\mathcal{A}_\kappa = \{A_\alpha: \alpha < \kappa\}$, where $A_\alpha = \{i: \exists p \in G ~ p(\alpha, i ) = 1\}$. Note that  $\mathcal{A}_\kappa$ is almost disjoint, and if $\kappa \geq \omega_1$ then it is maximal.

We start by giving a slightly modified definition of a non-linear forcing iteration $D(\omega, Q)$ from \cite{cmmngs}:

\begin{defi}\label{def_CS}
	Let $Q$ be a well-founded (strict) partial order such that $\aleph_1 \leq \mathfrak{b}(Q)$.
	Extend $Q$ to a partial order $Q'=Q\cup\{m\}$ with a maximal element $m$. Let $J \subseteq Q$ be fixed. For an arbitrary element $e \in Q'$ we denote the traces $Q'_e = \{c \in Q': c <_{Q'} e\}$ and $J_e = Q'_e \cap J$.  
	Recursively on $Q'$, define for each $a \in Q'$ a forcing notion $P_a$ as follows:
	\begin{itemize}
		\item  Fix $a\in Q'$ and suppose for each $b <_{Q'} a$ the poset $P_b$ has been defined and if $b\in Q'_a \backslash J$ then additionally we are given {\emph{a (nicely definable) $P_b$-name $\dot{H}_b$ for a family of functions in $^\omega\omega$.}}\footnote{Alternatively, we might require that we are given a $P_b$-name for a Suslin partial order $S_b$.}
		Then $P_a$ consists of functions $p$ such that $dom(p) \subseteq Q'_a$ and $|dom(p)| < \aleph_0$. Further $p$ maps each point of its domain either to a condition in the trivial forcing or to a condition in a restricted Hechler forcing depending on whether this point is in $J$ or not. That is for each $b \in dom(p)$ 
		
		$$p(b) = \begin{cases}
		\mathds{1} & b \in J \\
		(t, \dot{f})\in \mathbb{H}(\dot{H}_b) & b \not\in J
		\end{cases}$$
		Here $\mathbb{H}(\dot{H}_b)$ is a $P_b$-name  for the restricted Hechler forcing notion as above.\footnote{Alternatively, in the second case we might require that $p\upharpoonright b\Vdash_{P_b} p(b)\in \dot{S}_b$.} 
	\end{itemize}
	
	The extension relation of $P_a$ is defined as follows: if $p, q \in P_a$, then $p \leq q$ iff
	
	\begin{itemize}
		\item $dom(q) \subseteq dom(p)$ and
		
		\item $\forall b \in dom(q) \backslash J~[ p \upharpoonright b \Vdash_{P_b} p(b) \leq_{\mathbb{H}(\dot{H}_b)} q(b)]$\footnote{Alternatively $ p \upharpoonright b \Vdash_{P_b} p(b) \leq_{\mathbb{S}_b} q(b)$}, where $p \upharpoonright b = p \upharpoonright Q'_b$.
	\end{itemize}
	Finally, let $D(\omega, Q, J, \{\dot{H}_a\}_{a\in Q\backslash J})= P_m$.
\end{defi}

\begin{rem}
%$\bullet$ So, a good notation for the above poset seems to be $D(\omega, Q, J, \{\dot{H}_a\}_{a\in Q\backslash J})$ and in the more general case $D(\omega, Q, J,\{\dot{S}_a\}_{a\in Q\backslash J})$.

$\bullet$ To make the distinction between the poset $Q$ and the forcing notion \linebreak $D(\omega, Q, J, \{\dot{H}_a\}_{a\in Q\backslash J})$ clear, we refer to  the former as an index set (this will be one major function of the partial order $Q$ in our system of parallel non-linear iterations) and to the latter as a forcing notion.	

$\bullet$ In our intended application in the next section, for each $\alpha \leq \lambda$, the sets $J_\alpha$ and the family of names $\{\dot{H}^\alpha_a\}_{a\in Q\backslash J_\alpha}$ are decided by a bookkeeping function $F:Q\to\lambda$.
\end{rem}

We modify Lemma 10 in \cite{brndl/fschr} to non-linear iteration.

\begin{lem}\label{bs10}
	Let $R_0$, $R_1$ be c.c.c. partial orders such that $R_0\lessdot R_1$, $(Q, <_Q )= ([\delta]^{<\beta}, \subset)$ and $J^1\subseteq J^0 \subseteq Q$ be fixed. For each $l \in \{0,1\}$ and $c \in Q'$ let $T_{l,c}=D(\omega, Q_c , J^l_c, \{\dot{H}^l_a\}_{a\in Q_c\backslash J^l})$ be in $V^{R_l}$ for given families $\{\dot{H}^0_a\}_{a\in Q \backslash J^0}$ and $\{\dot{H}^1_a\}_{a\in Q \backslash J^1}$ of reals and $P_{l,c}=R_l*\dot{T}_{l,c}$.  Suppose $b \in Q'$ such that $rk_{Q'}(b) = \alpha$ is a limit ordinal and for each $a \in Q_b$, $P_{0, a}$ is a complete suborder of $P_{1, a}$ (denoted by $P_{0, a} \lessdot P_{1, a}$).
%\underline{and moreover for each $a\in Q_b \backslash J_0$ we have $\dot{H}^0_a=\dot{H}^1_a$.}
%\footnote{The requirement $\dot{H}^0_a=\dot{H}^1_a$ formulated more precisely states that $\dot{H}^1_a$ are a $P_{0,a}$ and $P_{1,a}$ names respectively for the same set of reals in $V_{0,a}=V^{P_{0,a}}$.} 
Then $P_{0, b} \lessdot P_{1, b}$.
\end{lem}

%and moreover for each $a\in Q_b \backslash J_0$ we have $\dot{H}^0_a=\dot{H}^1_a$, then $P_{0,b}\lessdot P_{1,b}$.\footnote{The requirement $\dot{H}^0_a=\dot{H}^1_a$ formulated more precisely states that $\dot{H}^1_a$ are a $P_{0,a}$ and $P_{1,a}$ names respectively for the same set of reals in $V_{0,a}=V^{P_{0,a}}$.}

\begin{proof}
	Let $\{b_i: i \in \alpha\}$ be a partition of $b$. Define $B_k = \bigcup\limits_{i \leq k}b_i$ for $k \in \alpha$; so $B_k <_Q b$ and $rk_{Q'}(B_k) < rk_{Q'}(b)$ for each $k \in \alpha$. By the finiteness of the supports and the assumption $P_{0, a} \lessdot P_{1, a}$ for all $a <_Q b$ it is clear that properties (1), (2), (3) of Definition \ref{compl.em} are satisfied. To finish the proof we find a reduction for every element. For this let $p \in P_{1, b}$ be arbitrary. Because $dom(p)$ is finite we have that $p \in P_{1, B_j}$ for some $j \in \alpha$. Since by assumption $P_{0, B_j} \lessdot P_{1, B_j}$ we can find a reduction $q \in P_{0, B_j}$  of $p$ and show that $q$ is also a reduction in $P_{0, b}$. So let $P_{0, b} \ni q' \not \perp q$ and let $P_{0, b} \ni r \leq q, q'$. Then write $r = r_0 \cup r_1$ while $r_0 \in P_{0, B_j}$ and $dom(r_1) \subseteq Q_b \setminus Q_{B_j}$. By $r \leq q$ we have $r_0 \leq q$ and because $q$ is a reduction of $p$ in $P_{0, B_j}$ there exists $r'_0 \in P_{0, B_j}$ such that $r'_0 \leq p,~ r_0$. Then $r'_0 \cup r_1 \leq r \leq q$ and $r'_0 \cup r_1 \leq p$, so we are done.
\end{proof}

%Part I of Lemma 11 in \cite{brndl/fschr} follows:

%\begin{lem}\label{bs11}
%Let $M \subseteq N$ be models of ZFC, $P \in M$ a forcing poset such that $P \subseteq M$, $G$ a $P$-generic filter over the greater model $N$ (hence also $P$-generic over $M$). Then the following property holds:

%Whenever $\mathcal{B} = \{B_\alpha\}_{\alpha < \gamma} \subseteq M \cap [\omega]^\omega$ and $A \in N  \cap [\omega]^\omega$ and $\bigstar(M, N, \mathcal{B}, A)$ is true, then \linebreak $\bigstar(M[G], N[G], \mathcal{B}, A)$ holds as well. 
%\end{lem}

\begin{defi}
	\begin{enumerate} Let $M$ and $N$ be models of set theory with $M \subseteq N$. 
		
		\item Let $s \in N \cap [\omega]^\omega$ be such that for each $a \in M \cap [\omega]^\omega$ we have $N \vDash |s \cap a| = \aleph_0 = |a \setminus s|$. Then $\clubsuit(M, N, s)$ holds.
		
		\item Let $u \in N \cap [\omega]^\omega$ such that for each $a \in M \cap [\omega]^\omega$ we have $N \vDash u \subseteq^* a$ or $u \subseteq^* (\omega\setminus a)$. Then $\blacklozenge(M, N, u)$ holds.
		
		\item Let $f \in N \cap \omega^\omega$ be such that for each predictor $\pi \in M$ we have $N \vDash f$ evades $\pi$. Then we say  that $\spadesuit(M, N, f)$ holds.
		
		\item Let $\pi \in N$ be a predictor such that for each $f \in M \cap \omega^\omega$, $N \vDash \pi$ predicts $f$. Then we say that $\heartsuit(M, N, \pi) $ holds.
	\end{enumerate}
	
\end{defi}

Let $\bullet \in \{ \clubsuit, \blacklozenge, \spadesuit, \heartsuit\}$. We need an analogous to Lemma 11 in \cite{brndl/fschr} for the $\bullet$-relation.

\begin{lem}\label{bs12}
	Let $M \subseteq N$ be models of ZFC, $P \in M$ a forcing poset such that $P \subseteq M$, $G$ a $P$-generic filter over the greater model $N$ (hence also $P$-generic over $M$). Then the following properties hold:
	
	(1) Whenever $s \in N  \cap [\omega]^\omega$ and $\clubsuit(M, N, s)$ is true, then $\clubsuit(M[G], N[G], s)$ holds as well. 
	
	(2) Whenever $u \in N  \cap [\omega]^\omega$ and $\blacklozenge(M, N, u)$ is true, then $\blacklozenge(M[G], N[G], u)$ holds as well. 
	
	(3) Whenever $f \in N  \cap \omega^\omega$ and $\spadesuit(M, N, f)$ is true, then $\spadesuit(M[G], N[G], f)$ holds as well.

	(4) Whenever $\pi \in N$ is a predictor and $\heartsuit(M, N, \pi)$ is true, then $\heartsuit(M[G], N[G], \pi)$ holds as well.
\end{lem}

\begin{proof}
	(1) For suppose not and let $x \in M[G] \cap [\omega]^\omega$ be such that $N[G] \vDash x \subseteq^* s$ or $x \subseteq^* (\omega\setminus s)$. By $\dot{x}$ we denote the $P$-name for $x$. In $M$ we will find a real $y$ which will contradict $\clubsuit(M, N, s)$. Let $x(k)$ denote the $k$-th value of $x$. As $N[G] \vDash \exists n \in \omega~ \forall m_0 \geq n~ [x(m_0) \in s] \lor \forall m_1 \geq n ~[x(m_1) \not\in s]$, there is a $p \in P$ and $n\in \omega$ with:
	$$p \Vdash_{N, P} \forall m_0 \geq n~[\dot{x}(m_0) \in s] \lor \forall m_1 \geq n~[\dot{x}(m_1) \not\in s].$$
	Further for $\dot{x}$ and $m \geq n$, let $p_m \in G$ be conditions deciding the $m$-th value of $x$, i.e. $p_m \Vdash_{M, P} \dot{x}(m) = k_m$. Now define the function 
	
	$$ f(m) = \begin{cases}
	m & if ~ m < n \\
	k_m & else
	\end{cases}.$$
	
	The function $f \in M$ and $y:= range(f) \in M \cap [\omega]^\omega$ contradicts the fact $\clubsuit(M, N, s)$.
	
	(2), (3), (4) are shown similar.
\end{proof}

	Now we modify Lemma 12 in \cite{brndl/fschr} to non-linear iteration to the right.

\begin{lem}\label{bs13}
	Let $b \in (Q, <_Q ) = ([\delta]^{<\beta}, \subset)$ be such that $rk_Q(b) = \omega$. As in Lemma \ref{bs10} let $R_0$, $R_1$ be c.c.c. partial orders such that $R_0\lessdot R_1$ and let $J^1\subseteq J^0 \subseteq Q$ be fixed. For each $l \in \{0,1\}$ and $c \in Q'$ let $T_{l,c}=D(\omega, Q_c , J^l_c, \{\dot{H}^l_a\}_{a\in Q_c\backslash J^l})$ be in $V^{R_l}$ for given families $\{\dot{H}^0_a\}_{a\in Q \backslash J^0}$ and $\{\dot{H}^1_a\}_{a\in Q \backslash J^1}$ and $P_{l,c}=R_l*\dot{T}_{l,c}$. Further suppose $\forall a \in Q_b~ [P_{0, a} \lessdot P_{1, a}]$. Let us denote $V_{l, a} = V^{P_{l, a}}$ for $l \in \{0,1\}$. Then the following are true.

	%(1) If $\mathcal{B} = \{A_\alpha\}_{\alpha < \gamma} \subseteq V_{0, 0 } \cap [\omega]^\omega$ and $A \in V_{1,0} \cap [\omega]^\omega$ and $\bigstar(V_{0, a}, V_{1, a}, \mathcal{B}, A)$ is true for each $a <_Q b$, then $\bigstar(V_{0, b}, V_{1, b}, \mathcal{B}, A)$ holds.
	
	(1) If $s \in V_{1,0}  \cap [\omega]^\omega$ and $\clubsuit(V_{0, a}, V_{1, a}, s)$ is true for each $a <_Q b$, then $\clubsuit(V_{0, b}, V_{1, b}, s)$ holds. 
	
	(2) If $u \in V_{1,0}  \cap [\omega]^\omega$ and $\blacklozenge(V_{0, a}, V_{1, a}, u)$ is true for each $a <_Q b$, then $\blacklozenge(V_{0, b}, V_{1, b}, u)$ holds. 
	
	(3) If $f \in V_{1,0}  \cap \omega^\omega$ and $\spadesuit(V_{0, a}, V_{1, a}, f)$ is true for each $a <_Q b$, then $\spadesuit(V_{0, b}, V_{1, b}, f)$ holds.
	
	(4) If $\pi \in V_{1,0}  \cap \omega^\omega$ is a predictor and $\heartsuit(V_{0, a}, V_{1, a}, \pi)$ is true for each $a <_Q b$, then $\heartsuit(V_{0, b}, V_{1, b}, \pi)$ holds.
\end{lem}

\begin{proof}

	(1) For suppose not and let $a \in V_{0, b} \cap [\omega]^{\omega}$ such that $V_{1, b} \vDash a \subseteq^* s \lor a \subseteq^*(\omega \setminus s)$. Then there is a condition in the generic, forcing this, i.e. there are $p \in G \subseteq P_{1, b}$, a $P_{0, b}$-name $\dot{a} \in V$ for $a$ and $m \in \omega$ with $p \Vdash \forall n_0 \geq m ~ [a(n_0) \in s] \lor \forall n_1 \geq m ~ [a(n_1) \not\in s]$. 
	
	Let $\{b_i: i \in \omega\}$ be a partition of $b$. Define $B_k = \bigcup\limits_{i \leq k}b_i$ for $k \in \omega$. As the domain of $p$ is finite there is a $j \in \omega$ such that $p \in P_{1, B_j}$. Now write $P_{l, b}$ as $P_{l, B_j} * \dot{R}^l_{B_j, b}$ for $l\in 2$ where $\dot{R}^l_{B_j, b}$ is the quotient poset and let $G_{1, B_j}$ be a $P_{1, B_j}$-generic filter containing $p$ and $G_{0, B_j} = G_{1, B_j} \cap P_{0, B_j}$ let $a'$ be the $R^0_{B_j, b}$-name in the model $V_{0, B_j}$ after partially evaluating the $P_{0, b}$-name $\dot{a}$ with respect to the generic $G_{0, B_j}$, i.e. the evaluation of $a'$ with respect to any $V_{0, B_j}$-generic $H \subseteq \dot{R}^l_{B_j, b}$ equals the evaluation of $\dot{a}$ with respect to the $V$-generic subset $G_{0, B_j}*H$ of $P_{0, B_j}*\dot{R}^l_{B_j, b} = P_{0, b}$, that is $a'[H] = \dot{a}[G_{0, B_j}*H]$.
	
	Then $a' \in V_{0, B_j}$ and we have in $V_{1, B_j}$
	
	\begin{center}
	$\Vdash_{R^0_{B_j, b}} \forall n_0 \geq m ~ [a'(n_0) \in s] \lor \forall n_1 \geq m ~ [a'(n_1) \not\in s].$
	\end{center}
	
	For each $n \geq m$ we find a condition deciding the $n$-th value of $a'$, i.e. there is $p_n \in {R^0_{B_j, b}}$ and $x_n \in \omega$ with $ p_{n} \Vdash a'(n) = x_{n}$. Now define the function $a_0$ as follows: $a _0\upharpoonright m = 0$ and $a_0(n) = x_{n}$ if $n \geq m$. Then $a_0 \in V_{0, B_j}$ and $a(n) \in s$ for each $n \geq m$ or $a(n) \not\in s$ for each $n \geq m$ contradicting $\clubsuit(V_{0, B_j}, V_{1, B_j}, s)$.
	
	(2), (3), (4) are shown similar.
\end{proof}

In the statement of the last lemma, $\omega$ is involved. We also need an analogous result for any stage of the non-linear iteration whose rank has countable cofinality. The proof is similar. For limits with uncountable cofinality the $\clubsuit$-, $\blacklozenge$-, $\spadesuit$ and $\heartsuit$-properties hold as the contradicting objects are countable.

Our next goal is to modify the well-known Lemma 13 in \cite{brndl/fschr}. We recall it.

\begin{lem*} Let $\mathbb{P}$ and $\mathbb{Q}$ be forcing notions with $\mathbb{P} \lessdot \mathbb{Q}$. Suppose $\dot{\mathbb{A}}$ (resp. $\dot{\mathbb{B}}$) is a $\mathbb{P}$-name (resp. $\mathbb{Q}$-name) for a forcing poset where $\Vdash_{\mathbb{B}} \dot{\mathbb{A}} \subseteq \dot{\mathbb{B}}$ and every maximal antichain of $\dot{\mathbb{A}}$ in $V^{\mathbb{P}}$ is a maximal antichain of $\dot{\mathbb{B}}$ in $V^{\mathbb{Q}}$. Then $\mathbb{P}*\dot{\mathbb{A}} \lessdot \mathbb{Q}*\dot{\mathbb{B}}$. 
\end{lem*}

In contrast to the linear case we have the following: If we pass over in a non-linear iteration to a stage $b$ with successor rank $\beta+1$, then the domain for the conditions is expanded not only by a single point, but by a multitude of points, namely $Q_b \setminus \bigcup\limits_{a <_Q b} Q_a$. As in the previous lemma, we want that the iterands behave "nicely" on the new part of the domain. The conditions (1)-(3) in Lemma \ref{bs16} aim to capture exactly this point.

\begin{rem}\label{REM_X}
Further note that these conditions will be satisfied in our construction in the next section as for each $a \in Q_b \backslash J_0$, hence particularly for those $a \in (Q_b \setminus \bigcup\limits_{a <_Q b} Q_a) \backslash J_0$, we will have\footnote{The equality $\dot{H}^0_a=\dot{H}^1_a$ formulated more precisely states that $\dot{H}^1_a$ are a $P_{0,a}$ and $P_{1,a}$ names respectively for the same set of reals in $V_{0,a}=V^{P_{0,a}}$.} $\dot{H}^0_a=\dot{H}^1_a$; and if $a \in J_0 \backslash J_1$ then we will trivially have $\{\mathds{1}\} \lessdot \mathbb{H}(\dot{H}^1_a)$.
\end{rem}

\begin{lem}\label{bs16}
Let $R_0$, $R_1$ be c.c.c. forcings such that $R_0\lessdot R_1$, $(Q, <_Q )= ([\delta]^{<\beta}, \subset)$ and $J^1\subseteq J^0 \subseteq Q$ be fixed. Suppose $b \in Q$ such that $rk_Q(b) = \beta +1$.
For each $l \in \{0,1\}$ and $c \in Q'$ let $T_{l,c}=D(\omega, Q_c , J^l_c, \{\dot{H}^l_a\}_{a\in Q_c\backslash J^l})$ be in $V^{R_l}$ for given families $\{\dot{H}^0_a\}_{a\in Q \backslash J^0}$ and $\{\dot{H}^1_a\}_{a\in Q \backslash J^1}$ and $P_{l,c}=R_l*\dot{T}_{l,c}$. Suppose $\forall a \in Q_b~ [P_{0, a} \lessdot P_{1, a}]$. Further assume that the families $\{\dot{H}^l_a\}_{a\in Q \backslash J^l}$ satisfy the following:
	\begin{enumerate}
		\item $\forall p, p' \in P_{0,b} \big[ [dom(p), dom(p') \subseteq Q_b \setminus \bigcup\limits_{a <_Q b} Q_a] \to [p \leq_{P_{0,b}} p' \to p \leq_{P_{1,b}} p'] \big]$
		
		\item $\forall p, p' \in P_{0,b} \big[ [dom(p), dom(p') \subseteq Q_b \setminus \bigcup\limits_{a <_Q b} Q_a] \to [p \perp_{P_{0,b}} p' \leftrightarrow p \perp_{P_{1,b}} p'] \big]$
		
		\item For all $p \in P_{1,b}$ with $dom(p) \subseteq Q_b \setminus \bigcup\limits_{a <_Q b} Q_a$ there exists a $q \in P_{0,b}$ such that $q$ is a reduction of $p$ to $P_{0,b}$. 
	\end{enumerate}
	Then $P_{0, b} \lessdot P_{1,b}$.
\end{lem}

\begin{proof}
	We use Lemma \ref{1.5} and check the conditions. (1) is clear.
	
	(2): Let $p, p' \in P_{0, b}$ and let $p \leq_{P_{0,b}} p'$. Let $( Q_b \setminus \bigcup\limits_{a <_Q b} Q_a) \cup \bigcup \{ b_a| a <_Q b, rk_Q(a)=\beta, b_a \subseteq Q_a \}$ be a partition of $Q_b$. Write $p = \bigcup\limits_{dom(r_a) \subseteq b_a} r_a \cup s, p' = \bigcup\limits_{dom(r'_a) \subseteq b_a} r'_a \cup s'$, where $dom(s), dom(s') \subseteq Q_b \setminus \bigcup\limits_{a <_Q b} Q_a$. Now for any $a$ we have
	$$r_a \leq_{P_{0, b}} r'_a  \underset{dom(r_a) \subseteq b_a}{\overset{dom(r'_a) \subseteq b_a}{\longrightarrow}} r_a \leq_{P_{0, a}} r'_a {\underset{P_{0,a} \lessdot P_{1,a}}{\longrightarrow}} r_a \leq_{P_{1, a}} r'_a {\underset{P_{1,a} \lessdot P_{1,b}}{\longrightarrow}} r_a \leq_{P_{1, b}} r'_a$$\\
	and also  $s \leq_{P_{1, b}} s'$ by assumption. Together this yields $p \leq_{P_{1, b}} p'$.
	
	(3): Let $p, p' \in P_{0, b}$ and let $p \perp_{P_{0,b}} p'$. Again write $p = \bigcup\limits_{dom(r_a) \subseteq b_a} r_a \cup s, p' = \bigcup\limits_{dom(r'_a) \subseteq b_a} r'_a \cup s'$, where $dom(s), dom(s') \subseteq Q_b \setminus \bigcup\limits_{a <_Q b} Q_a$ and we use the same partition as in (2). Now suppose there is an $a<_Q b$ such that $r_a \perp_{P_{0, b}} r'_a$ . Then 
	$$r_a \perp_{P_{0, b}} r'_a  \underset{dom(r_a) \subseteq b_a}{\overset{dom(r'_a) \subseteq b_a}{\longleftrightarrow}} r_a \perp_{P_{0, a}} r'_a {\underset{P_{0,a} \lessdot P_{1,a}}{\longleftrightarrow}} r_a \perp_{P_{1, a}} r'_a {\underset{P_{1,a} \lessdot P_{1,b}}{\longleftrightarrow}} r_a \perp_{P_{1, b}} r'_a;$$\\
	and because $s \perp_{P_{0, b}} s' \leftrightarrow s \perp_{P_{1, b}} s'$ by assumption, we have in total $p \perp_{P_{0, b}} p' \leftrightarrow p \perp_{P_{1, b}} p'$.
	
	Reduction: Let $p \in P_{1, b}$ and write $p = \bigcup\limits_{dom(r_a) \subseteq b_a} r_a \cup s$, where $dom(s) \subseteq Q_b \setminus \bigcup\limits_{a <_Q b} Q_a$ and we use the same partition as in (2). For any $a$ let $r'_a$ be a reduction of $r_a$ to $P_{0,a}$ and let $s'$ be a reduction of $s$ to $P_{0,b}$. Now define $p' = \bigcup\limits_{dom(r'_a) \subseteq b_a} r'_a \cup s'$ and check that this is a reduction of $p$ to $P_{0, b}$ as desired.
	So let $P_{0, b} \ni q \not \perp p'$ and let $ P_{0, b} \ni q' \leq p', q$. Then write $q' = \bigcup\limits_{dom(r''_a) \subseteq b_a} r''_a \cup s''$ while $dom(s'') \subseteq Q_b \setminus \bigcup\limits_{a <_Q b} Q_a$. By $q' \leq_{P_{0, b}} p'$ we have $r''_a \leq_{P_{0, a}} r'_a$ $s'' \leq_{P_{0, b}} s'$ and because the $r'_a$ and $s'$ are reductions of $r_a$ and $s$ (resp.) there exist $t_a, u$ such that $t_a \leq_{P_{0, a}} r_a, r''_a$ and $u \leq_{P_{0, b}} s, s''$. Then $\bigcup_a t_a \cup u \leq_{P_{0, b}} q' \leq_{P_{0, b}} p'$ and $\bigcup_a t_a \cup u \leq_{P_{0, b}} p$, so we are done.
\end{proof}

\section{A system of parallel non-linear iterations}

Now we are ready to construct our iteration. From now on let $\beta, \lambda, \delta$ be infinite cardinals with $\omega_1 \leq \beta = cf(\beta) \leq \lambda = cf(\lambda) \leq cf(\delta)$ be fixed uncountable cardinals and let $(Q, <_{Q} )= \linebreak ([\delta]^{< \beta}, \subset)$. For simplicity let $Q' = Q \cup \{m\}$, where $<_{Q'} \upharpoonright (Q \times Q) = <_{Q}$ and $\forall b \in Q~[ b <_{Q'} m]$. We first introduce a surjective book-keeping function $F: Q \to \lambda$ where $\forall \alpha < \lambda \forall b \in Q ~ [b \uparrow \cap F^{-1}(\alpha) \not= \emptyset]$. Recall that $|Q| = \delta \geq \lambda$ and for each $b \in Q$, $|b \uparrow| = \lambda$ hods. %Our matrix is defined recursively and consists of finite support iterations $\langle \langle P_{\alpha, \xi}: \alpha \leq \kappa, \xi \leq \lambda \rangle, \langle \dot{Q}_{\alpha, \xi}: \alpha \leq \kappa, \xi \leq \lambda \rangle \rangle$ where:

(1a) If $a = -1$ and $\lambda > \alpha = \beta +1 \equiv 0$ mod $4$, then $\Vdash_{P_{\beta, \emptyset}} \dot{\mathbb{Q}}_{\alpha, \emptyset} = \mathbb{C}$ where $\mathbb{C}$ denotes Cohen's forcing poset. Let $c_{\alpha+1}$ denote the Cohen real added at this stage.

(1b) If $a = -1$ and $\lambda > \alpha = \beta +1 \equiv 1$ mod $4$, then $\Vdash_{P_{\beta, \emptyset}} \dot{\mathbb{Q}}_{\alpha, \emptyset} = \mathbb{M}_{\dot{U}_{\beta, \emptyset}}$ from Definition \ref{MAT} where ${\dot{U}_{\beta, \emptyset}}$ is a $P_{\beta, \emptyset}$-name for an ultrafilter. Let $u_{\alpha+1}$ denote the Mathias real added at this stage.

(1c) If $a = -1$ and $\lambda > \alpha = \beta +1 \equiv 2$ mod $4$, then $\Vdash_{P_{\beta, \emptyset}} \dot{\mathbb{Q}}_{\alpha, \emptyset} = \mathbb{E}$ where $\mathbb{E}$ is the forcing poset from Definition \ref{EV}. Let $e_{\alpha+1}$ denote the real evading every predictor in the model $V_{\beta, \emptyset}$.

(1d) If $a = -1$ and $\lambda > \alpha = \beta +1 \equiv 3$ mod $4$, then $\Vdash_{P_{\beta, \emptyset}} \dot{\mathbb{Q}}_{\alpha, \emptyset} = \mathbb{P}$ where $\mathbb{P}$ is the forcing from Definition \ref{PR}. Let $\pi_{\alpha+1}$ denote the predictor, predicting every real added in the model $V_{\beta, \emptyset}$.

(1e) If $a = -1$ and $\lambda \geq \alpha$ is a limit stage, we take the direct limit.

For $\alpha \leq \lambda$ we denote by $P_{\alpha,-1}$ the above forcing iteration up to stage $\alpha$.

So far we performed a finite support iteration of length $\lambda$. The $\lambda$-many models form an increasing chain $\{V_{\alpha, -1}\}_{\alpha \in \lambda}$, which can be visualized as a column of models. Next we perform at each level of the column a non-linear iteration with (restricted) Hechler forcings. Note that at each level the underlying index set for the non-linear iteration is the same poset $Q$, so that the outcoming object can be visualized as a system of parallel "planes". Although the index set is the same at each level, the iterands are not. Seen this way, the following idea of adding restricted Hechler reals is a close analogue to the strategy in \cite{brndl/fschr}; the only difference is the non-linear index set.

(2) For each $\alpha \leq \lambda$ we define by induction on $\alpha \leq \lambda$ and by recursion on the well-founded poset $Q'$, $P_{\alpha, -1}$-names for forcing posets in $V_{\alpha, -1} = V^{P_{\alpha, -1}}$ as follows:
\begin{itemize}
		\item $J^\alpha=\{a\in Q: F(a)\geq\alpha\}$, and
		\item for each $a \in Q \backslash J^\alpha$, $\dot{H}^\alpha_a$ is a $P_{\alpha,a}$ name for  $V^{P_{F(a),a}}\cap{^\omega\omega}$.
	\end{itemize}
Then for each $b \in Q'$, define the forcing $T_{\alpha,b} = D(\omega, Q_b, J^\alpha_b, \{\dot{H}_a\}_{a\in Q_b \backslash J^\alpha_b})$ in $V_{\alpha,-1}$. Thus, the non-linear iteration in $V_{\alpha,-1}$ is uniquely determined by $P_{\alpha, -1}$ and $F$. In particular, we have 
$$V_{\alpha,-1}\vDash T_{\alpha,m}=D(\omega, Q, J^\alpha ,\{\dot{H}_a\}_{a\in Q\backslash J^\alpha}).$$

%in $\dot{T}_{\alpha, b}$ is defined as follows: if $\dot{p}, \dot{q} \in \dot{T}_{\alpha, b}$, then $\dot{p} \dot{\leq} \dot{q}$ iff 

%1) $dom(\dot{p}) \supseteq dom(\dot{q})$

%2a) for all $a \in dom(\dot{q})$ if $\dot{p}(a) = (s, \dot{h})$ and $\dot{q}(a) = (t, \dot{g})$, then 

%.) $t = s \upharpoonright dom(t)$

%.) $\dot{p}\upharpoonright a\downarrow ~\Vdash_{T_{\alpha, a}} dom(t) \leq i < dom(s) \to s(i) > \dot{g}(i)$

%.) $\dot{p}\upharpoonright a\downarrow ~\Vdash_{T_{\alpha, a}} \forall i \in \omega ~\dot{h}(i) \geq \dot{g}(i)$.
%(The condition for the second case where $\dot{q}(a) = \dot{p}(a) =\mathds{1}$ is empty as $\Vdash_{P_{\alpha, \emptyset} * T_{\alpha, a}} \mathds{1} \dot{\leq}\mathds{1}$.)

For $ b \in Q'$ define  $P_{\alpha, b} = P_{\alpha, -1} * T_{\alpha, b}$. Note that $J^\lambda = \emptyset$, so at the top "plane" we have no trivial forcings anymore, but only restricted Hechlers.

Further for our construction the following two properties will hold:

(a) $\forall b \in Q' \cup \{-1\} ~\forall \alpha < \alpha' \leq \lambda~ P_{\alpha, b} \lessdot P_{\alpha', b}$.

(b) $\forall b \in Q' \cup \{-1\} ~ \forall \alpha + 1 < \lambda~[\alpha + 1 \equiv 0$ mod $4 \to \clubsuit(V_{\alpha, b}, V_{\alpha+1, b}, c_{\alpha +1})]$.

(c) $\forall b \in Q' \cup \{-1\} ~ \forall \alpha + 1 < \lambda~[\alpha + 1 \equiv 1$ mod $4 \to\blacklozenge(V_{\alpha, b}, V_{\alpha+1, b}, u_{\alpha +1})]$.

(d) $\forall b \in Q' \cup \{-1\} ~ \forall \alpha + 1 < \lambda~[\alpha + 1 \equiv 2$ mod $4 \to\spadesuit(V_{\alpha, b}, V_{\alpha+1, b}, e_{\alpha +1})]$.

(e) $\forall b \in Q' \cup \{-1\} ~ \forall \alpha + 1 < \lambda~[\alpha + 1 \equiv 3$ mod $4 \to\heartsuit(V_{\alpha, b}, V_{\alpha+1, b}, \pi_{\alpha +1})]$.

To show that properties (a)-(e) indeed we proceed inductively. Let $b = -1$ and $\alpha \leq \lambda$, then (a) is satisfied by the properties of iterated forcing. The properties (b) and (c) are also satisfied as the Cohen forcing adds a splitting real, and the Mathias forcing adds an unsplit real.  Also properties (d) and (e) are satisfied as $\mathbb{E}$ adds evading real, and $\mathbb{P}$ adds a predictor, predicting every ground model real.

Let $b \in Q'$ then (a) is satisfied by induction on the rank in $Q'$ by the use of Lemma \ref{bs10} and Lemma \ref{bs16} and Remark \ref{REM_X}. Also (b)-(e) hold because of Lemma \ref{bs12} and Lemma \ref{bs13}.

For the purposes of the next remark and lemma define $\bar{Q} = Q' \cup \{-1\}$ and $\forall b \in Q'~[-1 <_{\bar{Q}} b]$ and $<_{\bar{Q}} \upharpoonright (Q' \times Q') = <_{Q'}$.

\begin{rem}\label{rem1}
	All together we have $\forall \alpha < \alpha' \leq \lambda~ \forall a <_{\bar{Q}} b \in \bar{Q} ~ P_{\alpha, a} \lessdot P_{\alpha', b}$.
\end{rem}

The next Lemma is analogous to Lemma 15 in \cite{brndl/fschr}.

\begin{lem}\label{abs}
	Suppose $b \in \bar{Q}$, then the following two properties hold:
	
	(a) Any condition $p \in P_{\lambda, b}$ is already in $P_{\alpha, b}$ for some suitable $\alpha < \lambda$.
	
	(b) If $\dot{f}$ is a $ P_{\lambda, b}$-name for a real then it is a $ P_{\alpha, b}$-name for a suitable $\alpha < \lambda$.
\end{lem}

\begin{proof}
	We show (a) and (b) simultaneously by transfinite recursion on $b \in \bar{Q}$, the well-founded poset. Because $P_{\lambda, b}$ has the c.c.c. property and $\lambda$ is regular uncountable we can easily see that (a) implies (b) if we pass over to a nice name of the real at hand. Now we begin the recursion by letting $b = -1$: Properties (a) and (b) for $b = -1$ are both true as $\lambda$ is regular uncountable and such a stage in a finite support iteration does not add new reals. 
	
	If $b \not= -1$ with $rk_{\bar{Q}}(b) = \gamma$ is a limit the claim is also true because of the finite supports any condition in $P_{\lambda, b}$ is already in some earlier $P_{\lambda, a}$ where the induction hypothesis holds, so (a) is true for stages with limit rank and implies (b) for stages with limit rank. 
	
	%Finally let $b \not= \emptyset$ with $rk_Q(b) = \gamma+1$. Let $(b\downarrow\setminus \bigcup\limits_{a <_Q b}a\downarrow) \cup \bigcup \{ b_a| a <_Q b, rk_Q(a)=\gamma, b_a \subseteq a \downarrow \}$ be a partition of $b \downarrow$. Then a condition $p \in P_{\beta, b}$ can be written as $p = \bigcup\limits_{dom(r_a) \subseteq b_a} r_a \cup s$, where $dom(s) \subseteq b\downarrow\setminus \bigcup\limits_{a <_Q b}a\downarrow$. Now $\beta$ is a limit and the domain of $p$ is finite. Hence so are the domains of $s$ and $r_a$. The images of $s$ at a point of its domain are countable objects and $\beta$ is regular uncountable. Now the induction hypothesis on (a) implies that each $r_a \in P_{\alpha_a, a}$ for some $\alpha_a < \beta$ and again the above mentioned well-known lemma implies that for each $c \in dom(s)$, $s(c) = p(c) \in P_{\alpha_c, b}$ and because the domains are finite there is an $\alpha \in \beta$ s.t. $p \in P_{\alpha, b}$.
	
	%Or:
	Finally let $b \not= -1$ with $rk_{\bar{Q}}(b) = \gamma+1$. Let an element $a <_{\bar{Q}} b$ partition $\bar{Q}_b$ into two parts: $\bar{Q}_a$ and $\bar{Q}_b \setminus \bar{Q}_a$ and $P_{\lambda, b} = P_{\lambda, a}*R_{\lambda, \bar{Q}_b \setminus \bar{Q}_a}$. (So $P_{\lambda, b}$ is split up into two parts: first forcing up to $a$ followed by the remaining non-linear iteration on the well-founded poset $\bar{Q}_b \setminus \bar{Q}_a$). Then a condition $p \in P_{\lambda, b}$ can be written as $p = (q_0, \dot{q}_1)$, where $q_0 \in P_{\lambda, a}$ and $\dot{q}_1$ is a $P_{\lambda, a}$-name for a condition in $R_{\lambda, \bar{Q}_b \setminus \bar{Q}_a}$. Now use the induction hypothesis on (a) and find an $\alpha_1 < \lambda$ with $q_0 \in P_{\alpha_1, b}$. The second part $\dot{q}_1$ involves only finitely many countable objects, so it can be coded as a real and then by the use of the induction hypothesis on (b) we can find an $\alpha_2 < \lambda$ such that $\dot{q}_1$ is a $P_{\alpha_2, b}$-name. Hence $p \in P_{\alpha_3, b}$ where $\alpha_3 = max\{\alpha_1, \alpha_2\}$.
\end{proof}

The next lemma gives us the consistency result.

\begin{lem}
	$V_{\lambda, m} \vDash  \mathfrak{b} = \beta \leq \mathfrak{r} =\mathfrak{e} = \mathfrak{s} = \lambda \leq  \delta = \mathfrak{d}$.
\end{lem}

\begin{proof}
	$\mathfrak{s} \leq \lambda$: The family $\mathcal{S}:= \{c_\alpha: ~ \lambda > \alpha \equiv 0$ mod $ 4\}$ added in the first column is a splitting family in the model $V_{\lambda, m}$. If this was not the case, then $\exists x \in V_{\lambda, m} \cap [\omega]^\omega ~ \forall c_\alpha \in \mathcal{S} ~ x \subseteq^* c_\alpha \lor x \subseteq^* (\omega \setminus c_\alpha)$. By Lemma \ref{abs} we have $\exists \alpha <  \lambda~ \alpha \equiv 3$ mod $4 \land x \in V_{\alpha, m} \cap [\omega]^\omega$. However on the other side we have $\clubsuit(V_{\alpha, m}, V_{\alpha+1, m}, c_{\alpha+1})$ meaning that $c_{\alpha+1}$ splits $x$.
	
	$\mathfrak{r} \leq \lambda$: The family $\mathcal{U}:= \{u_\alpha: ~ \lambda > \alpha \equiv 1$ mod $ 4\}$ added in the first column is a reaping family in the model $V_{\lambda, m}$. If this was not the case, then $\exists x \in V_{\lambda, m} \cap [\omega]^\omega ~ \forall u_\alpha \in \mathcal{U} ~ |u_\alpha \setminus x|= \aleph_0 = |x \cap u_\alpha|$. By Lemma \ref{abs} we have $\exists \alpha <  \lambda~ \alpha \equiv 0$ mod $4 \land x \in V_{\alpha, m} \cap [\omega]^\omega$. However on the other side we have $\blacklozenge(V_{\alpha, m}, V_{\alpha+1, m}, u_{\alpha+1})$ meaning that $u_{\alpha+1}$ is not split by $x$.
	
	$\mathfrak{e} \leq \lambda$: The family $\mathcal{E}:= \{e_\alpha: ~ \lambda > \alpha \equiv 2$ mod $ 4\}$ added in the first column is not predicted by a single predictor in the model $V_{\lambda, m}$. If this was not the case, then there would be a predictor $ \pi \in V_{\lambda, m}$ such that  $\forall e_\alpha \in \mathcal{S} ~ \pi$ predicts $e_\alpha$. As a predictor is also a countable object Lemma \ref{abs} implies $\exists \alpha <  \lambda~ \alpha \equiv 1$ mod $4 \land \pi \in V_{\alpha, m}$. However on the other side we have $\spadesuit(V_{\alpha, m}, V_{\alpha+1, m}, e_{\alpha+1})$ meaning that $e_{\alpha+1}$ evades $\pi$.
	
	$\mathfrak{s} \geq \lambda$: Let $\mathcal{A}$ be a set of reals in the final model $V_{\lambda, m}$ such that $|\mathcal{A}| < \lambda$. By Lemma \ref{abs} and the regularity of $\lambda$ we have $\exists \alpha <  \lambda~ \alpha \equiv 0$ mod $4 \land \mathcal{A} \subseteq V_{\alpha, m} \cap [\omega]^\omega$. However we have $\blacklozenge(V_{\alpha, m}, V_{\alpha+1, m}, u_{\alpha+1})$ meaning that $u_{\alpha+1}$ is not split by any element in $\mathcal{A}$, hence $\mathcal{A}$ is not splitting.
	
	$\mathfrak{r} \geq \lambda$: Let $\mathcal{A}$ be a set of reals in the final model $V_{\lambda, m}$ such that $|\mathcal{A}| < \lambda$. By Lemma \ref{abs} and the regularity of $\lambda$ we have $\exists \alpha <  \lambda~ \alpha \equiv 3$ mod $4 \land \mathcal{A} \subseteq V_{\alpha, m} \cap [\omega]^\omega$. However we have $\clubsuit(V_{\alpha, m}, V_{\alpha+1, m}, c_{\alpha+1})$ meaning that $s_{\alpha+1}$ splits any alement in $\mathcal{A}$, hence $\mathcal{A}$ is not reaping.
	
	$\mathfrak{e} \geq \lambda$: Let $\mathcal{A}$ be a set of reals in the final model $V_{\lambda, m}$ such that $|\mathcal{A}| < \lambda$. By Lemma \ref{abs} and the regularity of $\lambda$ we have $\exists \alpha <  \lambda~ \alpha \equiv 2$ mod $4 \land \mathcal{A} \subseteq V_{\alpha, m} \cap \omega^\omega$. However we have $\heartsuit(V_{\alpha, m}, V_{\alpha+1, m}, \pi_{\alpha+1})$ meaning that $\pi_{\alpha+1}$ predicts any alement in $\mathcal{A}$, hence $\mathcal{A}$ is not a witness for $\mathfrak{e}$.
	
	By the previous paragraphs we have $V_{\lambda, m} \vDash \mathfrak{s} = \mathfrak{r} = \mathfrak{e} = \lambda$.
	
	$\mathfrak{b} \geq \beta$: Let $B \subseteq V_{\lambda, m} \cap {}^\omega\omega$ be such that $|B|<\beta$. Since $\mathfrak{b}(Q) = \beta$ and by Lemma \ref{abs} we have $\exists b \in Q, \alpha < \lambda ~ B \subseteq V_{\alpha, b} \cap {}^\omega\omega$. As $\forall \alpha < \lambda \forall b \in Q~[b \uparrow \cap F^{-1}(\alpha) \not= \emptyset]$ we can find an element $b < b' \in Q$ with $F(b') = \alpha$. Then the poset $P_{\alpha +1, b'\cup \{max(b')+1\}}$ adds a dominating real over $V_{\alpha, b'} \cap {}^\omega\omega \supseteq V_{\alpha, b} \cap {}^\omega\omega$, hence $B$ is not unbounded.

	$\delta \geq \mathfrak{d}$: Let $\dot{f}$ be a $P_{\lambda, m}$-name for a real. By the previous Lemma \ref{abs} and $\mathfrak{b}(Q) = \beta \geq \aleph_1$ and $\lambda$ is regular uncountable, there is a $b \in Q, \alpha < \lambda ~ f \in V_{\alpha, b} \cap {}^\omega\omega$. Let $D \subseteq Q$ be a dominating family of size $\delta$ and let $d \in D$ be such that $b <_Q d$.  As $\forall \alpha < \lambda \forall b \in Q ~ [b \uparrow \cap F^{-1}(\alpha) \not= \emptyset]$ we can find an element $d < b^d_\alpha \in Q$ with $F(b^d_\alpha) = \alpha$. Then $P_{\alpha +1, b^d_\alpha \cup \{max(b^d_\alpha)+1\}}$ adds a dominating real over the model $V_{\alpha, b^d_\alpha} \supseteq V_{\alpha, b} $, call it $g^{b^d_\alpha}$. Hence the arbitrary $f$ is dominated by the set $\{g^{b^d_\alpha}: d \in D, \alpha \in \lambda\}$ which is of size $\delta*\lambda = \delta$.

	Now let $G$ be a $V_{\lambda, -1}$-generic subset of $T_{\lambda, m}$ and let $f^a_G =  \bigcup\{t(a): \exists p \in G~ [p(a) = (t(a), \dot{f}(a))]\}$.
	
	\begin{clm}
		If $g \in V_{F(a), a}$ and $b \not<_Q a$, then $f^b_G \not \leq^* g$.
	\end{clm}
	
	\begin{proof}
		Let $p$ be an arbitrary condition in $P_{\lambda, m}$ and $n \in \omega$, we will find an extension of $p$ which forces $f^b_G(k) > g(k)$ for some $k \geq n$. Let $p(a) = (t, \dot{g}')$ and $p(b) = (s, \dot{h})$. Let $\dot{g}$ be a $P_{\lambda, a}$-name for $g$. Let $\dot{f}$ be a $P_{\lambda, a}$-name for the pointwise maximum of $\dot{g}'$ and $\dot{g}$. Now define the condition $p_0$ as follows: $dom(p_0) = dom(p)$ and $p_0(c) = p(c)$ for each $c \not= a$, and $p_0(a) = (t, \dot{f})$. Clearly $p_0 \leq p$. Now let $k \in \omega$ be large enough such that $\{dom(t), dom(s), n \} \subset k$. Next let $q \in P_{\lambda, a}$ extend $p_0 \upharpoonright a$ (= $p_0 \upharpoonright Q_a$) and $q$ decides the value of $\dot{f}$ up to $k$. Now define the extension $p_1$ of $p_0$ by setting $p_1(c) = p_0(c)$ for each $c \not<_Q a$ and $p_1(c) = q(c)$ for each $c <_Q a$. So $p_1$ is a extension of $p_0$ carrying the information on the values of $\dot{f}$ up to $k$; and now we do the same for $b$ and $p_1$, so we let $r \in P_{\lambda, b}$ with $r \leq p_1 \upharpoonright b$ and $r$ decides the values of $\dot{h}$ up to $k$. We define the extension $p_2$ as $p_2(c) = p_1(c)$ for each $c \not<_Q b$ and $p_2(c) = r(c)$ for each $c <_Q b$. Now $p \geq p_0 \geq p_1 \geq p_2$ and $p_2(a) = p_0(a)$ and $p_2(b) = p(b)$. Now we extend $p_2$ as desired: First find a final extension $t' \supseteq t$ such that $dom(t') = k+1$ and for $dom(t) \leq i < dom(t') ~ t'(i)> \dot{f}(i)$. Then find a final extension $s' \supseteq s$ such that $dom(s') = k+1$ and for $dom(s) \leq i  < k+1 ~ [s'(i)> max\{\dot{h}(i), t'(i)\}]$. Then an extension satisfying the latter forces $f^b_G(k) > f(k)$ which gives the claim.
	\end{proof}
	
	$\beta \geq \mathfrak{b}$: Let $U \subseteq Q$ be an unbounded family of size $\beta$ and let $f \in V_{\lambda, m} \cap {}^\omega\omega$. As $f$ is a countable object and $\mathfrak{b}(Q) = \beta \geq \aleph_1$, there is an $a \in Q$ such that $f \in V_{\lambda, a} \cap {}^\omega\omega$. By Lemma \ref{abs} $f \in V_{\alpha, a} \cap {}^\omega\omega$. The book-keeping function $F$ ensures that $f \in V_{F(b), b} \cap {}^\omega\omega$ for some $a <_Q b$. As $U$ is unbounded $\exists u \in U~ u\not<_Q b$. Then by the last claim $f^u_G \not \leq^* f$. Hence $\{f^u_G: u \in U\}$ is an unbounded family of size $\beta$.
	
	$\delta \leq \mathfrak{d}$: Let $F \subseteq V_{\lambda, m} \cap {}^\omega\omega$ be a family of size less than $\delta$. As in the previous paragraph we can find for every single $f \in F$ a stage $a_f \in Q$ such that $f \in V_{F(a_f), a_f} \cap {}^\omega\omega$. Now $|\{a_f: f \in F\}| < \delta$, so $\{a_f: f \in F\}$ is not dominating in $Q$. Hence $\exists u \in Q \forall f \in F~ [u \not<_Q a_f]$. Then by the last claim $\forall f \in F~[ f^u_G \not \leq^* f]$. Hence $F$ is not dominating.
\end{proof}

The next theorem follows:

\begin{thm}\label{MT}
	If $\beta, \lambda, \delta, \mu$ are infinite cardinals with $\omega_1 \leq \beta = cf(\beta) \leq \lambda = cf(\lambda) \leq \delta \leq \mu$ and $cf(\mu) > \omega$, then there is a c.c.c. generic extension of the ground model in which $\beta = \mathfrak{b} \land \lambda = \mathfrak{r} = \mathfrak{s} = \mathfrak{e} \land \delta = \mathfrak{d} \land \mu = \mathfrak{c}$ holds.
\end{thm}

\begin{proof}
	In the above construction replace the underlying poset $(Q, <_{Q} )= ([\delta]^{< \beta}, \subset)$ by the following poset $(R, <_{R})$: $R$ consists of pairs $(p, i)$ such that either $ i = 0 \land p \in \mu$ or $i = 1 \land p \in Q$. The order relation is defined as $(p,i) <_{R} (q,j)$ iff $i = 0 \land j = 1$ or $i = j = 1 \land p <_{Q} q$ or $i = j= 0 \land p < q$ in $\mu$.
\end{proof}

As it is shown in \cite{aubrey} $\mathfrak{r} < \mathfrak{d}$ implies $\mathfrak{u} = \mathfrak{r}$. So if we choose $\lambda < \delta$ in the statement of Theorem \ref{MT}, then also $\mathfrak{u}$ is determined in the model of Theorem \ref{MT}, giving the following corollary.

\begin{cor}
	If $\beta, \lambda, \delta, \mu$ are infinite cardinals with $\omega_1 \leq \beta = cf(\beta) \leq \lambda = cf(\lambda) < \delta \leq \mu$ and $cf(\mu) > \omega$, then  $Con( \beta = \mathfrak{b} \land \lambda = \mathfrak{r} = \mathfrak{u} = \mathfrak{s} = \mathfrak{e} \land \delta = \mathfrak{d} \land \mu = \mathfrak{c})$.
\end{cor}

\section{Remarks and Questions}\label{questions_section}

Four candidates, namely $\mathfrak{s}, \mathfrak{r}, \mathfrak{e}$ and $\mathfrak{u}$, for being between 
$\mathfrak{b}$ and $\mathfrak{d}$ are in fact controlled between these two invariants. One of them (among others) remained open. 

\begin{question}
Is it relatively consistent to have $\mathfrak{b} < \mathfrak{a} < \mathfrak{d} < \mathfrak{c}$?
\end{question}


\begin{thebibliography}{00}
	
	%\bibitem{drdl} J. E. Baumgartner/ P. Dordal: Adjoining Dominating Functions, The Journal of Symbolic Logic, Vol. 50, No. 1, 1985, pp. 94-101.
	
		\bibitem{aubrey} Aubrey, Jason: Combinatorics for the Dominating and Unsplitting Numbers, The Journal of Symbolic Logic, 1 June 2004, Vol.69(2), pp.482-498.
	
	\bibitem{ofb/vf} \"O. F. Ba\u{g}, V. Fischer: Non-linear iterations and higher splitting, 2020, preprint.
	
	\bibitem{blass}
	Blass, Andreas: Combinatorial Cardinal Characteristics of the Continuum; in: M. Foreman/A. Kanamori [Ed.]: Handbook of Set Theory, Springer, Heidelberg/London/New York, 2010, pp. 395-489.
	
	\bibitem{brndl/fschr}
	J. Brendle/ V. Fischer: Mad families, splitting families and large continuum, The Journal of Symbolic Logic, Vol.76(1), 1 March 2011, pp.198-208. 
	
	\bibitem{brndl/shlh}
	J. Brendle/ S. Shelah: Evasion and Prediction II. Journal of the London Mathematical Society, Vol. 53(1), February 1996, pp. 19–27.
	
	\bibitem{cmmngs} J. Cummings/ S. Shelah: Cardinal invariants above the continuum, Annals of Pure and Applied Logic 75, 1995, pp. 251-268.
		
	\bibitem{kun} Kunen, Kenneth: Set Theory, College Publications, London, 2013.
	
	\bibitem{vnDwn} van Douwen, Eric K.: The integers and topology. Handbook of set-theoretic topology, North-Holland, Amsterdam/New York, 1984, pp. 111-167.
	
	
	
	%\bibitem[Jech]{jch} Jech, Thomas: Set Theory. The Third Millennium Edition, Springer, Berlin/Heidelberg/New York, 2006.
		
	
	
\end{thebibliography}
\end{document}